\documentclass[12pt]{article}
\usepackage{amsmath,amssymb}
\usepackage{amsthm}
\usepackage{enumitem}
\usepackage[toc]{appendix}
\usepackage{multicol}
\usepackage{comment}
\usepackage{hyperref}
\usepackage{url}
\usepackage{graphicx}
\usepackage{xcolor}
\usepackage{color}

\newtheorem{theorem}{Theorem}
\newtheorem{lemma}[theorem]{Lemma}
\newtheorem{proposition}[theorem]{Proposition}
\newtheorem{corollary}[theorem]{Corollary}

\theoremstyle{definition}
\newtheorem{definition}[theorem]{Definition}
\newtheorem{observation}[theorem]{Observation}
\newtheorem{remark}[theorem]{Remark}
\newtheorem{example}[theorem]{Example}
\newtheorem{problem}{Problem}
\newtheorem{question}[theorem]{Question}

\newenvironment{thm}{\begin{theorem}}{\end{theorem}}
\newenvironment{lem}{\begin{lemma}}{\end{lemma}}

\newenvironment{cor}{\begin{corollary}}{\end{corollary}}

\def\1{{\bf 1}}

\title{Unified bounds for the independence number of graph powers}

\author{Aida Abiad\thanks{\texttt{a.abiad.monge@tue.nl}, Department of Mathematics and Computer Science, Eindhoven University of Technology, The Netherlands} \thanks{Department of Mathematics and Data Science, Vrije Universiteit Brussel, Belgium} \qquad Jiang Zhou\thanks{\texttt{zhoujiang@hrbeu.edu.cn}, College of Mathematical Sciences, Harbin Engineering University, Harbin 150001, PR China}}

\date{}

\begin{document}
\maketitle

\begin{abstract}
For a graph $G$, its $k$-th power $G^k$ is constructed by placing an edge between two vertices if they are within distance $k$ of each other. The $k$-independence number $\alpha_k(G)$ is defined as the independence number of $G^k$. By using general semidefinite programming and polynomial methods, we derive sharp bounds for the $k$-independence number of a graph, which extend and unify various existing results. Our work also allows us to easily derive some new  bounds for $\alpha_k(G)$.\\


\end{abstract}

%%%%%%%%%%%%%%%%%%%%%%%%%%%%%%%%%%%%%%%%%%%%%%%%%%%%%%%%%%%
\section{Introduction}
%%%%%%%%%%%%%%%%%%%%%%%%%%%%%%%%%%%%%%%%%%%%%%%%%%%%%%%%%%%

We consider simple undirected graphs. Let $G$ be a graph with vertex set $V(G)$ and edge set $E(G)$. Let $d(i,j)$ denote the distance between vertices $i$ and $j$. A \emph{$k$-independent set} in a graph $G$ is a vertex subset $U\subseteq V(G)$ such that $d(i,j)>k$ for any two distinct vertices $i,j\in U$. The \emph{$k$-independence number} of $G$, denoted by $\alpha_k(G)$, is the size of the largest $k$-independent set in $G$. The $k^{\text{th}}$ \emph{power graph} of $G$, denoted by $G^k$, is the graph with the same vertices as $G$ where two vertices are adjacent if they are at distance at most $k$ in $G$. Alternatively, $\alpha_k(G)=\alpha(G^k)$, and $\alpha_1(G)=\alpha(G)$ equals the independence number of $G$. Graph powers are useful in designing efficient algorithms for certain combinatorial optimization problems, see, e.g., \cite{2,5}. In distributed computing, the $k$-th power of graph $G$ represents the possible flow of information during $k$ rounds of communication in a distributed network of processors organized according to $G$ \cite{17}. For other work involving graph powers, see, e.g., \cite{1, 9,13,15,16}.

For an $n$-vertex $d$-regular graph $G$, the well-known Hoffman ratio bound says that $\alpha(G)\leq n\frac{|\tau|}{d-\tau}$, where $\tau$ is the minimum eigenvalue of $G$. Several generalized Hoffman ratio bounds, some of which use Laplacian eigenvalues, have appeared in the literature, see \cite{DH1998,GN,H1995,LLT2007,Z2004}. The Lov\'{a}sz theta function $\vartheta(G)$, which was introduced in \cite{Lovasz1979}, gives an upper bound for $\alpha(G)$ which is known to be better than the Hoffman bound. The Lov\'{a}sz theta function is a powerful tool for studying the Shannon capacity $\Theta(G)$. Indeed, Lov\'{a}sz \cite{Lovasz1979} proved that $\alpha(G)\leq\Theta(G)\leq\vartheta(G)$. The Schrijver theta function $\vartheta'(G)$, introduced in \cite{Schrijver}, is an even better upper bound for $\alpha(G)$ since it holds that $\alpha(G)\leq\vartheta'(G)\leq\vartheta(G)$. In \cite{z2023}, $\vartheta(G)$, $\vartheta'(G)$ and various generalized Hoffman ratio bounds were unified into one form.

Since~$\alpha_k(G) = \alpha(G^k)$, most bounds on the independence number can be applied directly to upper bound the $k$-independence number. However, to use the known eigenvalue bounds for $\alpha$ on $G^k$, such as the Hoffman ratio bound, one needs to know the spectrum of $G^k$, which in general cannot be derived from the spectrum of $G$ itself. Therefore, Abiad et al. \cite{acf2019} proposed an inertia- and ratio-type bound on the $k$-independence number that only depend on the spectrum of $G$. These inertia- and ratio-type bounds, which are known to be sharp for several graph classes, use a degree-$k$ polynomial which can be chosen freely. The quality of such bounds depends on the choice of the polynomial, so finding the best possible upper bound for a given graph is in fact an optimization problem. For the case of $k$-partially walk-regular graphs, the optimal polynomials for the optimization of the ratio-type bound was studied by Fiol \cite{fiol20}. For general graphs, the optimal polynomials for the inertia-type bound was shown by Abiad et al. \cite{acfns2022}. Recently, the ratio-type bound on $\alpha_k(G)$ was used to estimate the maximum size of a code in a certain metric, often outperforming the state of the art coding bounds, see \cite{ANR2024,AKR2024}.

In this paper we present general semidefinite programming (SDP) and polynomial methods to establish unified bounds on $\alpha_k(G)$, and we discuss its applications by deriving several existing results in the literature (such as the ratio-type bound). Our framework extends the recent results by the second author \cite{z2023}. Moreover, as a byproduct of our general method, we obtain several of the mentioned known bounds on $\alpha_k(G)$ such as \cite[Theorem 3.2]{acf2019} (or alternatively \cite[Corollary 28]{AZ2024}), \cite[Theorem 4.1]{fiol20} and \cite[Theorem 4.3]{GN}. Our framework also allows us to easily derive some new bounds for $\alpha_k(G)$.

This paper is organized as follows. In Section \ref{sec:preiminaries}, we establish the used notation and introduce some preliminary lemmas. In Section 3, we give unified bounds on $\alpha_k(G)$, $\vartheta(\overline{G^k})$ and $\vartheta'(\overline{G^k})$ in terms of graph matrices of $G$. Based on the results in Section \ref{sec:frameworkunification}, we derive various old and new algebraic bounds on $\alpha_k(G)$ in Section \ref{sec:consequences}. We end up with some concluding remarks in Section \ref{sec:concludingremarks}.

%%%%%%%%%%%%%%%%%%%%%%%%%%%%%%%%%%%%%%%%%%%%%%%%%%%%%%%%%%%
\section{Preliminaries}\label{sec:preiminaries}
%%%%%%%%%%%%%%%%%%%%%%%%%%%%%%%%%%%%%%%%%%%%%%%%%%%%%%%%%%%

For a graph $G$ with vertex set $V(G)$ and edge set $E(G)$, the \textit{adjacency matrix} $A_G$ of $G$ is a $|V(G)|\times|V(G)|$ symmetric matrix with entries
\begin{eqnarray*}
(A_G)_{ij}=\begin{cases}1~~~~~~~~~~\mbox{if}~\{i,j\}\in E(G),\\
0~~~~~~~~~~\mbox{if}~\{i,j\}\notin E(G).\end{cases}
\end{eqnarray*}
Eigenvalues of the adjacency matrix of $G$ are called \textit{eigenvalues} of $G$. 

For a real square matrix $M$, the \textit{group inverse} of $M$, denoted by $M^\#$, is the matrix $X$ such that $MXM=M,~XMX=X$ and $MX=XM$. It is known that $M^\#$ exists if and only if $\mbox{\rm rank}(M)=\mbox{\rm rank}(M^2)$. If $M^\#$ exists, then $M^\#$ is unique. Let ${\rm tr}~M$ denote the trace of $M$.

For a positive semidefinite real matrix $M$, there exists an orthogonal matrix $U$ such that
\begin{eqnarray*}
M=U{\rm diag}(\lambda_1,\ldots,\lambda_n)U^\top,
\end{eqnarray*}
where ${\rm diag}(\lambda_1,\ldots,\lambda_n)$ denotes a diagonal matrix with nonnegative diagonal entries $\lambda_1,\ldots,\lambda_n$. Then
\begin{eqnarray*}
M^\#=U{\rm diag}(\lambda_1^+,\ldots,\lambda_n^+)U^\top,
\end{eqnarray*}
where $\lambda_i^+=\lambda_i^{-1}$ if $\lambda_i>0$, and $\lambda_i^+=0$ if $\lambda_i=0$.
\begin{lem}\label{lema:1}
\textup{\cite{z2023}} Let $M$ be a square matrix such that $M^\#$ exists, and let $\lambda\neq0$ be an eigenvalue of $M$ with an eigenvector $x$. Then
\begin{eqnarray*}
M^\#x=\lambda^{-1}x.
\end{eqnarray*}
\end{lem}

Let $R(M)=\{x:x=My,y\in\mathbb{R}^n\}$ denote the range of an $m\times n$ real matrix $M$. A real vector $x=(x_1,\ldots,x_n)^\top$ is called \textit{total nonzero} if $x_i\neq0$ for $i=1,\ldots,n$.

For a graph $G$, let $\mathcal{M}(G)$ denote the set of real matrix-vector pairs $(M,x)$ such that\\
(a) $M$ is a positive semidefinite matrix indexed by vertices of $G$ such that $(M)_{ij}=0$ if $i,j$ are two nonadjacent distinct vertices.\\
(b) $x\in R(M)$ is a total nonzero vector.
\begin{lem}\label{lema:2}
\textup{\cite[Theorem 3.1]{z2023}} Let $(M,x)\in\mathcal{M}(G)$ be a matrix-vector pair associated with an $n$-vertex graph $G$. Then the following statements hold:\\
(1) For any independent set $S$ of $G$, we have
\begin{eqnarray*}
|S|^2\leq x^\top M^\#x\sum_{u\in S}\frac{(M)_{uu}}{x_u^2}.
\end{eqnarray*}
(2) The independence number $\alpha(G)$ satisfies
\begin{eqnarray*}
\alpha(G)\leq x^\top M^\#x\max_{u\in V(G)}\frac{(M)_{uu}}{x_u^2}.
\end{eqnarray*}
\end{lem}

For a graph $G=(V(G),E(G))$, let $G^{\square t}$ denote the graph whose vertex set is $V(G)^t$, in which two vertices $u_1\cdots u_t$ and $v_1\cdots v_t$ are adjacent if and only if for each $i\in\{1,\ldots,t\}$ either $u_i=v_i$ or $\{u_i,v_i\}\in E(G)$. The \textit{Shannon capacity} of $G$ is defined as
\begin{eqnarray*}
\Theta(G)=\sup_{t}\alpha(G^{\square t})^{1/t}.
\end{eqnarray*}

For a matrix $A$ over a field $\mathbb{F}$, let ${\rm rank}_\mathbb{F}(A)$ denote the rank of $A$ over $\mathbb{F}$. For a field $\mathbb{F}$ and a graph $G$, let $\mathcal{A}_\mathbb{F}(G)$ denote the set of graph matrices $A$ over $\mathbb{F}$ satisfying the following two conditions:
\begin{description}
    \item[$(a)$] $A$ is an $|V(G)|\times|V(G)|$ matrix indexed by the vertices of $G$.
     \item[$(b)$] $(A)_{ij}=0$ if $i\neq j$ and $\{i,j\}\notin E(G)$.
\end{description}

\begin{lem}\cite{H1979}\label{lema:minrankG}\textup{\cite[Theorem 3.7.8]{BH}}
Let $A\in\mathcal{A}_\mathbb{F}(G)$ be a matrix associated with an $n$-vertex graph $G$. If $(A)_{uu}\neq0$ for each $u\in V(G)$. Then
\begin{eqnarray*}
\alpha(G)\leq\Theta(G)\leq{\rm rank}_\mathbb{F}(A).
\end{eqnarray*}
\end{lem}

An \textit{orthonormal representation} of an $n$-vertex graph $G$ is a set of unit real vectors $\{u_1,\ldots,u_n\}$  such that $u_i^\top u_j=0$ if $i$ and $j$ are two nonadjacent vertices in $G$. The \textit{value} of an orthonormal representation $\{u_1,\ldots,u_n\}$ is defined to be
\begin{eqnarray*}
\min_c\max_{1\leq i\leq n}\frac{1}{(c^\top u_i)^2},
\end{eqnarray*}
where $c$ ranges over all unit real vectors. The \textit{Lov\'{a}sz theta function} $\vartheta(G)$ \cite{Lovasz1979} is the minimum value of all orthonormal representations of $G$.

\begin{lem}\label{lema:4}
\textup{\cite[Theorem 4]{Lovasz1979}} Let $G$ be a graph with vertex set $\{1,\ldots,n\}$, and let $e$ be the all-ones vector of dimension $n$. Then
\begin{eqnarray*}
\vartheta(G)=\max_Be^\top Be,
\end{eqnarray*}
where $B$ range over all positive semidefinite real matrices of order $n$ such that ${\rm tr}~B=1$ and $(B)_{ij}=0$ for every pair $\{i,j\}\in E(G)$.
\end{lem}

Let $G$ be a graph with vertex set $\{1,\ldots,n\}$, and let $e$ be the all-ones vector of dimension $n$. In \cite{Schrijver}, the Schrijver theta function $\vartheta'(G)$ is defined as
\begin{eqnarray*}
\vartheta'(G)=\max_Be^\top Be,
\end{eqnarray*}
where $B$ range over all positive semidefinite nonnegative matrices of order $n$ such that ${\rm tr}~B=1$ and $(B)_{ij}=0$ for every pair $\{i,j\}\in E(G)$. Schrijver proved that $\alpha(G)\leq\vartheta'(G)\leq\vartheta(G)$ (see \cite[Theorem 1]{Schrijver}).

%%%%%%%%%%%%%%%%%%%%%%%%%%%%%%%%%%%%%%%%%%%%%%%%%%%%%%%%%%%%%%%%%%%%%%%%%%%%%%%%%%%%%%%%%%%%%%%%%%%%%%%%%
\section{A framework to unify bounds on $\alpha_k$}\label{sec:frameworkunification}
%%%%%%%%%%%%%%%%%%%%%%%%%%%%%%%%%%%%%%%%%%%%%%%%%%%%%%%%%%%%%%%%%%%%%%%%%%%%%%%%%%%%%%%%%%%%%%%%%%%%%%%%%

\textcolor{blue}{}

A partition $V(G)=V_1\cup\cdots\cup V_t$ is called a \textit{$k$-distance coloring} of a graph $G$ if $V_1,\ldots,V_t$ are $k$-independent sets in $G$. Actually, a $k$-distance coloring of $G$ is a proper coloring of $G^k$.

For an $n$-vertex graph $G$, let $\mathcal{M}_k(G)$ denote the set of real matrix-vector pairs $(M,x)$ such that\\
(a) $M$ is a positive semidefinite $n\times n$ matrix indexed by vertices of $G$ such that $(M)_{ij}=0$ if $d(i,j)>k$.\\
(b) $x=(x_1,\ldots,x_n)^\top\in R(M)$ is a total nonzero vector.

Using Lemma 2 we can then obtain the following result.
\begin{thm}\label{thm1}
Let $(M,x)\in\mathcal{M}_k(G)$ be a matrix-vector pair associated with an $n$-vertex graph $G$. Then the following statements hold:\\
(1) For any $k$-independent set $S$ of $G$, we have
\begin{eqnarray*}
|S|^2\leq x^\top M^\#x\sum_{u\in S}\frac{(M)_{uu}}{x_u^2}.
\end{eqnarray*}
(2) The $k$-independence number $\alpha_k(G)$ satisfies
\begin{eqnarray*}
\alpha_k(G)\leq x^\top M^\#x\max_{u\in V(G)}\frac{(M)_{uu}}{x_u^2}.
\end{eqnarray*}
(3) For any $k$-distance coloring $V(G)=V_1\cup\cdots\cup V_t$, we have
\begin{eqnarray*}
\sum_{i=1}^t|V_i|^2\leq x^\top M^\#x\sum_{u\in V(G)}\frac{(M)_{uu}}{x_u^2}.
\end{eqnarray*}
\end{thm}

\begin{proof}
Since $\mathcal{M}_k(G)=\mathcal{M}(G^k)$, parts (1) and (2) follow from Lemma \ref{lema:2}. Part (1) implies part (3).
\end{proof}

Let $\mathbb{R}_k[x]$ denote the set of polynomials of real coefficients and degree at most $k$. For a graph $G$, let $\mathcal{S}(G)$ be the set of real symmetric matrices $A$ such that $(A)_{ij}=0$ if $d(i,j)>1$.
\begin{thm}\label{thm:6}
Let $A\in\mathcal{S}(G)$ be a matrix associated with a graph $G$ such that $A$ has an eigenvalue $\lambda$ with a total nonzero eigenvector $y$. Let $p(x)\in \mathbb{R}_k[x]$ be a polynomial such that $p(A)$ is positive semidefinite and $p(\lambda)>0$. Then
\begin{eqnarray*}
\alpha_k(G)\leq\frac{y^\top y}{p(\lambda)}\max_{u\in V(G)}\frac{(p(A))_{uu}}{y_u^2}.
\end{eqnarray*}
\end{thm}

\begin{proof}
Since $Ay=\lambda y$, we have $p(A)y=p(\lambda)y$. So $y\in R(p(A))$. From the definition of $A$, we know that $(p(A))_{ij}=0$ whenever $d(i,j)>k$. Hence $(p(A),e)\in\mathcal{M}_k(G)$. By Lemma \ref{lema:1}, we have
\begin{eqnarray*}
p(A)^\#y=p(\lambda)^{-1}y,\\
y^\top p(A)^\#y=\frac{y^\top y}{p(\lambda)}.
\end{eqnarray*}
By Theorem \ref{thm1}, we have
\begin{align*}
&\alpha_k(G)\leq\frac{y^\top y}{p(\lambda)}\max_{u\in V(G)}\frac{(p(A))_{uu}}{y_u^2}.
\qedhere\end{align*}
\end{proof}
Let $\overline{G}$ denote the complement of a graph $G$.
\begin{thm}\label{thm:7}
Let $A\in\mathcal{S}(G)$ be a matrix associated with a graph $G$, and let $p(x)\in \mathbb{R}_k[x]$ be a polynomial such that $p(A)$ is positive semidefinite and ${\rm tr}~p(A)>0$. Then
\begin{eqnarray*}
\vartheta(\overline{G^k})\geq\frac{e^\top p(A)e}{{\rm tr}~p(A)},
\end{eqnarray*}
where $e$ is the all-ones vector. Furthermore, if $p(A)$ is nonnegative, then
\begin{eqnarray*}
\vartheta'(\overline{G^k})\geq\frac{e^\top p(A)e}{{\rm tr}~p(A)}.
\end{eqnarray*}
\end{thm}

\begin{proof}
Let $B=({\rm tr}~p(A))^{-1}p(A)$, then $B$ is positive semidefinite with ${\rm tr}~B=1$. Notice that $(B)_{ij}=0$ for every pair $\{i,j\}\in E(\overline{G^k})$. By Lemma \ref{lema:4}, we have
\begin{eqnarray*}
\vartheta(\overline{G^k})\geq e^\top Be=\frac{e^\top p(A)e}{{\rm tr}~p(A)}.
\end{eqnarray*}
If $p(A)$ is nonnegative, then by the definition of $\vartheta'(\overline{G^k})$, we have
\begin{align*}
&\vartheta'(\overline{G^k})\geq\frac{e^\top p(A)e}{{\rm tr}~p(A)}.
\qedhere\end{align*}
\end{proof}

We can derive the following result from Theorems \ref{thm:6} and \ref{thm:7}.

\begin{thm}\label{thm2}
Let $A\in\mathcal{S}(G)$ be a matrix associated with an $n$-vertex graph $G$ such that $A$ has constant row sum $d$. Let $p(x)\in \mathbb{R}_k[x]$ be a polynomial such that $p(A)$ is positive semidefinite and $p(d)>0$. Then
\begin{eqnarray*}
\alpha_k(G)&\leq&\frac{n}{p(d)}\max_{u\in V(G)}(p(A))_{uu},\\
\vartheta(\overline{G^k})&\geq&\frac{np(d)}{{\rm tr}~p(A)}.
\end{eqnarray*}
\end{thm}

\begin{proof}
Let $e$ be the all-ones vector, then $Ae=de$ and $p(A)e=p(d)e$. By Theorem \ref{thm:6}, we have
\begin{align*}
&\alpha_k(G)\leq \frac{n}{p(d)}\max_{u\in V(G)}(p(A))_{uu}.
\qedhere\end{align*}
By $p(d)>0$, we know that ${\rm tr}~p(A)>0$. By Theorem \ref{thm:7}, we have
\begin{eqnarray*}
\vartheta(\overline{G^k})\geq\frac{e^\top p(A)e}{{\rm tr}~p(A)}=\frac{np(d)}{{\rm tr}~p(A)}.
\end{eqnarray*}
\end{proof}
A \emph{walk} (of length $k$) in a graph $G$ is an alternating sequence 
\begin{eqnarray*}
u_1e_1u_2e_2\cdots u_ke_ku_{k+1}
\end{eqnarray*}
such that $v_i,v_{i+1}$ are two distinct end-vertices of the edge $e_i$ ($i=1,\ldots,k$). If $u_1=u_{k+1}$, then this walk is called a closed walk. Let $w_k(G)$ denote the number of walks of length $k$ in $G$, and let $c_k(G)$ denote the number of closed walks of length $k$ in $G$.
\begin{cor}
Let $G$ be a graph with at least one edge. For any positive integer $k$, we have
\begin{eqnarray*}
\vartheta'(\overline{G^{2k}})\geq\frac{w_{2k}(G)}{c_{2k}(G)}.
\end{eqnarray*}
\end{cor}

\begin{proof}
Let $A$ be the adjacency matrix of $G$. For any positive integer $t$, it is well known that $(A^t)_{ij}$ equals to the number of walks of length $t$ from the vertex $i$ to the vertex $j$. Then $A^{2k}$ is a positive semidefinite nonegative matrix and $c_{2k}(G)={\rm tr}~A^{2k}>0$. Let $e$ be the all-ones vector. By Theorem \ref{thm:7}, we have
\begin{align*}
&\vartheta'(\overline{G^{2k}})\geq\frac{e^\top A^{2k}e}{{\rm tr}~A^{2k}}=\frac{w_{2k}(G)}{c_{2k}(G)}.
\qedhere\end{align*}
\end{proof}

%%%%%%%%%%%%%%%%%%%%%%%%%%%%%%%%%%%%%%%%%%%%%%%%%%%%%%%%%%%%%%%%%%%%%%%%%%%%%%%%%%%%%%%%%%%%%%%%%%%%%%%%%
\section{Consequences}\label{sec:consequences}
%%%%%%%%%%%%%%%%%%%%%%%%%%%%%%%%%%%%%%%%%%%%%%%%%%%%%%%%%%%%%%%%%%%%%%%%%%%%%%%%%%%%%%%%%%%%%%%%%%%%%%%%%

%%%%%%%%%%%%%%%%%%%%%%%%%%%%%%%%%%%%%%%%%%%%%%%%%%%%%%%%
\subsection{An SDP bound on $\alpha_k$}
%%%%%%%%%%%%%%%%%%%%%%%%%%%%%%%%%%%%%%%%%%%%%%%%%%%%%%%%

The spectral radius of the adjacency matrix of a graph $G$ is called the \textit{spectral radius} of $G$. The \textit{principal eigenvector} of a connected graph $G$ is the unique positive eigenvector $y$ corresponding to the spectral radius of $G$ such that $\sum_{u\in V(G)}y_u^2=1$.

\begin{cor}\label{coro:alphaksdpbound}
Let $G$ be a connected graph with spectral radius $\rho$ and principal eigenvector $y$. Let $p(x)\in \mathbb{R}_k[x]$ be a polynomial such that $p(A)$ is positive semidefinite and $p(\rho)>0$, where $A$ is the adjacency matrix of $G$. Then
\begin{align*}
&\alpha_k(G)\leq p(\rho)^{-1}\max_{u\in V(G)}\frac{(p(A))_{uu}}{y_u^2}.
\qedhere\end{align*}
\end{cor}

\begin{proof}
Let $A$ be the adjacency matrix of $G$, then $Ay=\rho y$ and $y^\top y=1$. By Theorem \ref{thm:6}, we have
\begin{align*}
&\alpha_k(G)\leq p(\rho)^{-1}\max_{u\in V(G)}\frac{(p(A))_{uu}}{y_u^2}.
\qedhere\end{align*}
\end{proof}

In order to optimize the bound from Corollary \ref{coro:alphaksdpbound}, we would like to find the polynomial $p$ which minimizes the right-hand side for a given graph $G$. Some observations:
\begin{itemize}
    \item The expression is invariant under multiplying $p$ by a positive scalar. We may therefore assume that $\max_{u\in V(G)} \frac{(p(A))_{uu}}{y_{uu}^2} = 1$.
    \item Maximizing over $u$ is equivalent to requiring $(p(A))_{uu}/y_{uu}^2 \le 1$ for all $u\in V(G)$. This gets rid of the $\min\max$.
    \item The objective function then simplifies to $\min p(\rho)^{-1}$, which we can find by maximizing $p(\rho)$.
\end{itemize}

We therefore solve the following SDP.

\[
\boxed{
\begin{array}{rl}
{\tt maximize} & p(\rho) \\
{\tt subject\ to} & (p(A))_{uu}\le y^2_{uu}\quad u\in V\\
& p(\rho) > 0\\
&p(A) \succcurlyeq 0
\end{array}
}
\]

The value of the bound in Corollary~\ref{coro:alphaksdpbound} then equals one divided by the optimal value of the SDP.

%%%%%%%%%%%%%%%%%%%%%%%%%%%%%%%%%%%%%%%%%%%%%%%%%%%%%%%
\subsection{Algebraic bounds on $\alpha_k$}
%%%%%%%%%%%%%%%%%%%%%%%%%%%%%%%%%%%%%%%%%%%%%%%%%%%%%%%%

From Theorem \ref{thm2}, we can derive \cite[Theorem 3.2]{acf2019} (see also  \cite[Corollary 28]{AZ2024}) as follows.
\begin{cor}[Ratio-type bound, 
\textup{\cite{acf2019}}] \label{coro:ratiotypebound}Let $G$ be a regular graph with $n$ vertices and adjacency eigenvalues $\lambda_1\geq\cdots\geq\lambda_n$. Let $p(x)\in \mathbb{R}_k(x)$ be a polynomial such that $p(\lambda_1)>\lambda(p)=\min_{2\leq i\leq n}\{p(\lambda_i)\}$. Then
\begin{eqnarray*}
\alpha_k(G)\leq n\frac{\max_{u\in V(G)}(p(A))_{uu}-\lambda(p)}{p(\lambda_1)-\lambda(p)}.
\end{eqnarray*}
\end{cor}

\begin{proof}
Let $A$ be the adjacency matrix of $G$, then $A$ has constant row sum $\lambda_1$ since $G$ is regular. Let $q(x)=p(x)-\lambda(p)$. Since $p(\lambda_1)>\lambda(p)=\min_{2\leq i\leq n}\{p(\lambda_i)\}$, we know that $q(\lambda_1)=p(\lambda_1)-\lambda(p)>0$ and $q(A)=p(A)-\lambda(p)I$ is positive semidefinite. 

Finally, by Theorem \ref{thm2}, we have
\begin{align*}
&\alpha_k(G)\leq \frac{n}{q(\lambda_1)}\max_{u\in V(G)}(q(A))_{uu}.
\qedhere\end{align*}
\end{proof}

The optimization of the above ratio-type bound for general $k$ was investigated by Fiol \cite{fiol20}, and its tightness was studied by Abiad et al \cite{acf2019}. For $k=3$ the best polynomial for Corollary \ref{coro:ratiotypebound} was obtained by Kavi and Newman \cite{kavi_optimal_2023}.

A graph $G$ is called \textit{$k$-partially walk-regular}, for some integer $k\geq 0$, if the number of closed walks of a given length $l\leq k$, rooted at a vertex $v$, only depends on $l$.

Let $G$ be a graph with adjacency matrix $A$ and eigenvalues $\lambda_1\geq\cdots\geq \lambda_n$, and let $\mathcal{P}_k=\{f(x)\in \mathbb{R}_k[x]:f(\lambda_1)=1,f(\lambda_i)\geq0~for~2\leq i\leq n\}$. The \textit{$k$-minor polynomial} of $G$ is the polynomial $f_k\in\mathcal{P}_k$ such that ${\rm tr}~f_k(A)=\min\{{\rm tr}~f(A):f\in\mathcal{P}_k\}$.

From Theorem \ref{thm2}, we can derive \cite[Theorem 4.1]{fiol20} as follows.
\begin{cor}
\textup{\cite{fiol20}} Let $G$ be a $k$-partially walk-regular graph with adjacency matrix $A$, let $f_k(x)\in \mathbb{R}_k[x]$ be a $k$-minor polynomial. Then
\begin{eqnarray*}
\alpha_k(G)\leq{\rm tr}~f_k(A).
\end{eqnarray*}
\end{cor}

\begin{proof}
Let $\lambda_1\geq\cdots\geq\lambda_n$ be the adjacency eigenvalues of $G$, then $A$ has constant row sum $\lambda_1$, and $f_k(\lambda_1),\ldots,f_k(\lambda_n)$ are eigenvalues of $f_k(A)$. Since $f_k(\lambda_1)=1$ and $f_k(\lambda_i)\geq0$ for $2\leq i\leq n$, $f_k(A)$ is positive semidefinite.

Since $G$ is $k$-partially walk-regular, by Theorem \ref{thm2}, we have
\begin{align*}
&\alpha_k(G)\leq n\max_{u\in V(G)}(f_k(A))_{uu}={\rm tr}~f_k(A).
\qedhere\end{align*}
\end{proof}

The optimization and tightness of the above bound was investigated by Fiol in \cite{fiol20}.

The \emph{Laplacian matrix} of a graph $G$ is defined as $L=D-A$, where $D$ is the diagonal matrix of vertex degrees of $G$, $A$ is the adjacency matrix of $G$. For a vertex subset $S$, let $d_k(S)=|S|^{-1}\sum_{u\in S}(L^k)_{uu}$. Eigenvalues of $L$ are called Laplacian eigenvalues of $G$. 

Some inequalities involving independent sets and Laplacian eigenvalues can be found in \cite[inequality (6.15)]{M1991} and \cite[Theorem 4.3]{GN}. By using the largest Laplacian eigenvalue and our general framework from Section \ref{sec:frameworkunification}, we can also obtain the following upper bound for a $k$-independent set $S$ of a graph.
\begin{thm}\label{thm4}
Let $G$ be an $n$-vertex graph with largest Laplacian eigenvalue $\mu>0$. For any $k$-independent set $S$ of $G$, we have
\begin{eqnarray*}
|S|\leq\frac{n(\mu^k-d_k(S))}{\mu^k}.
\end{eqnarray*}
\end{thm}

\begin{proof}
Let $e$ be the all-ones vector, then $(\mu^kI-L^k)e=\mu^ke$. So $e\in R(\mu^kI-L^k)$. Notice that $(L^k)_{ij}=0$ whenever $d(i,j)>k$. Hence $(\mu^kI-L^k,e)\in\mathcal{M}_k(G)$. By Lemma \ref{lema:1}, we have
\begin{eqnarray*}
(\mu^kI-L^k)^\#e=\mu^{-k}e,\\
e^\top(\mu^kI-L^k)^\#e=\frac{n}{\mu^k}.
\end{eqnarray*}
By Theorem \ref{thm1}, we have
\begin{align*}
|S|^2&\leq e^\top(\mu^kI-L^k)^\#e\sum_{u\in S}(\mu^kI-L^k)_{uu},\\
|S|&\leq \frac{n(\mu^k-d_k(S))}{\mu^k}.
\qedhere\end{align*}
\end{proof}
Taking $k=1$ in Theorem \ref{thm4}, we obtain \cite[Theorem 4.3]{GN}.
\begin{cor}
\textup{\cite{GN}} Let $G$ be an $n$-vertex graph with largest Laplacian eigenvalue $\mu>0$. For any independent set $S$ of $G$, we have
\begin{eqnarray*}
|S|\leq\frac{n(\mu-\overline{d}_S)}{\mu},
\end{eqnarray*}
where $\overline{d}_S$ is the average degree of all vertices in $S$.
\end{cor}

%%%%%%%%%%%%%%%%%%%%%%%%%%%%%%%%%%%%%%%%%%%%%%%%%%%%%%%%
\subsection{Bound for the distance-$k$ coloring}
%%%%%%%%%%%%%%%%%%%%%%%%%%%%%%%%%%%%%%%%%%%%%%%%%%%%%%%%

For the $k$-distance coloring of a regular graph, we can derive the following inequality from the new framework.
\begin{thm}
Let $G$ be an $n$-vertex $d$-regular graph with adjacency matrix $A$. Let $p(x)\in \mathbb{R}_k[x]$ be a polynomial such that $p(A)$ is positive semidefinite and $p(d)>0$. For any $k$-distance coloring $V(G)=V_1\cup\cdots\cup V_t$, we have
\begin{eqnarray*}
\sum_{i=1}^t|V_i|^2\leq\frac{n}{p(d)}{\rm tr}~p(A).
\end{eqnarray*}
\end{thm}

\begin{proof}
Let $e$ be the all-ones vector, then $Ae=de$ and $p(A)e=p(d)e$. So $e\in R(p(A))$ and $(p(A),e)\in\mathcal{M}_k(G)$. By Lemma 1, we have
\begin{eqnarray*}
p(A)^\#e=p(d)^{-1}e,\\
e^\top p(A)^\#e=\frac{n}{p(d)}.
\end{eqnarray*}
By Theorem \ref{thm1}, we have
\begin{align*}
&\sum_{i=1}^t|V_i|^2\leq\frac{n}{p(d)}{\rm tr}~p(A).
\qedhere\end{align*}
\end{proof}

%%%%%%%%%%%%%%%%%%%%%%%%%%%%%%%%%%%%%%%%%%%%%%%%%%%%%%%%%%%%%%%%%%%%%%%%%%%%%%%%%%%%%%%%%%%%%%%%%%%%%%%%%
\subsection{A minimum rank-type bound on $\alpha_k$}
%%%%%%%%%%%%%%%%%%%%%%%%%%%%%%%%%%%%%%%%%%%%%%%%%%%%%%%%%%%%%%%%%%%%%%%%%%%%%%%%%%%%%%%%%%%%%%%%%%%%%%%%%

Let $\mathbb{F}_k[x]$ denote the set of polynomials with coefficients in a field $\mathbb{F}$ and degree at most $k$.

\begin{thm}
Let $A\in\mathcal{A}_\mathbb{F}(G)$ be a matrix associated  with an $n$-vertex graph $G$. Let $p(x)\in \mathbb{F}_k[x]$ be a polynomial such that $(p(A))_{uu}\neq0$ for each $u\in V(G)$. Then
\begin{eqnarray*}
\alpha_k(G)\leq\Theta(G^k)\leq{\rm rank}_\mathbb{F}(p(A)).
\end{eqnarray*}
\end{thm}

\begin{proof}
Note that $(p(A))_{uv}=0$ if $d(u,v)>k$. By Lemma \ref{lema:minrankG}, we have
\begin{align*}
&\alpha_k(G)=\alpha(G^k)\leq\Theta(G^k)\leq{\rm rank}_\mathbb{F}(p(A)).
\qedhere\end{align*}
\end{proof}

The tightness and optimization of the above bound is currently being investigated in \cite{ADF2024}.

%%%%%%%%%%%%%%%%%%%%%%%%%%%%%%%%%%%%%%%%%%%%%%%%%%%%%%%%%%%
\section{Concluding remarks}\label{sec:concludingremarks}
%%%%%%%%%%%%%%%%%%%%%%%%%%%%%%%%%%%%%%%%%%%%%%%%%%%%%%%%%%%

The proposed framework to unify existing and to derive several new algebraic bounds on $\alpha_k$ has a direct application in coding theory. Indeed, several of the eigenvalue bounds which follow form our framework (such as Corollary \ref{coro:ratiotypebound})  have recently been shown to be very useful for bounding the maximal size of a code with a certain minimum distance in several metrics, see \cite{ANR2024,AKR2024}.

The \emph{distance-$k$ chromatic number} $\chi_k(G)$ of a graph $G$ is defined as the chromatic number of $G^k$. Since $\chi_k(G)\geq \frac{n}{\alpha_k(G)}$, our framework yields lower bounds for the distance-$k$ chromatic number of a graph. Moreover, since the \emph{distance-$k$ chromatic index} $\chi_k'(G)$ satisfies
 $$\chi_k'(G) = \chi_k(L(G)) = \chi_1(L(G)^k),$$
where $L(G)$ is the line graph of $G$, then our framework also provides lower bounds for the distance-$k$ chromatic index.

From \cite[Theorem 4.1]{z2023}, we know that the Lov\'asz theta function $\vartheta$ of the power graph holds
\begin{eqnarray*}
\vartheta(G^k)=\min_{(M,x)\in\mathcal{M}_k(G)}x^\top M^\#x\max_{u\in V(G)}\frac{(M)_{uu}}{x_u^2}.
\end{eqnarray*}

Since using \cite{Lovasz1979} we know that $\Theta(G^k)\leq\vartheta(G^k)$, then all algebraic bounds on $\alpha_k(G)$ that we derived in this paper are also upper bounds of $\Theta(G^k)$ and $\vartheta(G^k)$.

In \cite{F1997}, Fiol (improving work from \cite{FG1998}) obtained the following eigenvalue bound on $\alpha_k(G)$.
\begin{theorem}\cite{F1997}\label{thm{F1997alternatingpolys}} If $G$ is a regular graph with adjacency eigenvalues $\lambda_1\geq \cdots \geq \lambda_n$, then
$$\alpha_k(G) \leq \frac{2n}{P_k(\lambda_1)},$$
where $P_k$ is the $k$-alternating polynomial of $G$. 
\end{theorem}

This result was later generalized to nonregular graphs in \cite{F1999}. The polynomial $P_k$ is defined by the solution of a linear
programming problem which depends on the spectrum of the graph $G$. Moreover, despite the bound from Theorem \ref{thm{F1997alternatingpolys}} uses an analogue proof technique as the obtained Corollary \ref{coro:ratiotypebound}, it seems not to follow from the framework proposed in this work.

%%%%%%%%%%%%%%%%%%%%%%%%%%%%%%%%%%%%%%%%%%%%%%%%%%%%%%%%%%%
\subsection*{Acknowledgments}
%%%%%%%%%%%%%%%%%%%%%%%%%%%%%%%%%%%%%%%%%%%%%%%%%%%%%%%%%%%
Aida Abiad is supported by the Dutch Research Council through the grant VI.Vidi.213.085. Jiang Zhou is supported by the National Natural Science Foundation of China (No. 12071097), and the Natural Science Foundation for The Excellent Youth Scholars of the Heilongjiang Province (No. YQ2022A002).

%%%%%%%%%%%%%%%%%%%%%%%%%%%%%%%%%%%%%%%%%%%%%%%%%%%%%%%%%%%

\end{document}